\documentclass[12pt]{amsart}
\usepackage{graphicx, subfigure, epsfig, epsf, verbatim}
\usepackage{amsmath} 
\usepackage{amssymb} 
\usepackage{amsfonts} 

\newtheorem{theorem}{Theorem}
\newtheorem{lemma}[theorem]{Lemma}
\newtheorem{corollary}[theorem]{Corollary}
\newtheorem{Pro}{Proposition}
\theoremstyle{remark}

\newtheorem{remark}{Remark}
\newtheorem*{acknowledgement}{Acknowledgment}

\newcommand{\NN}{\mathbb{N}}
\newcommand{\ZZ}{\mathbb{Z}}

 \newcommand{\CC}{\mathbb{C}}
 
\newcommand{\TT}{\mathbb{T}} 
\newcommand{\EE}{\mathbb{E}} 
\newcommand{\PP}{\mathbb{P}}
\newcommand{\RR}{\mathbb{R}} 

\newcommand{\iy}{\infty}
\newcommand{\lcap}{\operatorname{cap}}
\newcommand{\ellcap}{\operatorname{lcap}}
\newcommand{\diam}{\operatorname{diam}}
\newcommand{\dx}{\textrm{d}}

\begin{document}
\bibliographystyle{alpha}

\title[Scaling limits of anisotropic Hastings-Levitov clusters]{Scaling limits of anisotropic Hastings-Levitov clusters}

%\date{\today}

\author[Johansson]{Fredrik Johansson}
\address{Johansson: Department of Mathematics, Royal Institute of Technology, 
100 44 Stockholm, Sweden}
\email{frejo@math.kth.se}

\author[Sola]{Alan Sola}
\address{Sola: Department of Mathematics, Royal Institute of Technology, 
100 44 Stockholm, Sweden}
\email{alansola@math.kth.se}

\author[Turner]{Amanda Turner}
\address{Turner: Department of Mathematics and Statistics, Fylde College, 
Lancaster University, Lancaster, LA1 4YF, U.K.}
\email{a.g.turner@lancaster.ac.uk}


\thanks{Johansson is 
supported by grant KAW 2005.0098 from the Knut and Alice Wallenberg 
Foundation. \\
Sola is 
supported by grant KAW 2005.0098 from the Knut and Alice Wallenberg 
Foundation.}

\begin{abstract}
We consider a variation of the standard Hastings-Levitov model ${\rm HL}(0)$, in which growth is anisotropic. Two natural scaling limits are established and we give precise descriptions of the effects of the anisotropy. We show that the limit shapes can be realised as Loewner hulls and that the evolution of harmonic measure on the cluster boundary can be described by the solution to a deterministic ordinary differential equation related to the Loewner equation. We also characterise the stochastic 
fluctuations around the deterministic limit flow.
\end{abstract}
\maketitle

\section{Introduction} 
\subsection{Generalized $\rm{HL}(0)$ clusters}
In this paper we consider growing sequences of compact 
sets in the complex plane $\CC$ obtained by composing random conformal mappings. 
Let $D_0$ denote the exterior unit disk
\[D_0=\{z \in \CC_{\iy}: |z|>1\},\]
and let $K_0=\CC \setminus D_0$ be the closed unit disk.
We consider a simply connected set $D_1 \subset D_0$, such that  
$P=D_1^c \setminus K_0$ has diameter $d \in (0,1]$ and $1 \in \overline{P}$.
%Let $P$ be a closed, connected, simply connected subset of $D_0$ of diameter $d \in(0,1]$
%such that $P\cap K_0=\{1\}$. We set $D_1=\CC_\infty\setminus(K_0\cup P)$.
The set $P$ models an incoming particle, which is attached to the unit
disk at $1$. 
There exists a unique conformal mapping
\begin{equation}
f_{P}: D_0 \rightarrow D_1
\end{equation}
with expansion at infinity of the form
$f_{P}(z)=C(P)z+o(z)$ for some $C(P)>0$. The value $C(P)=\lcap(K_0 \cup P)$ is called the logarithmic
capacity.

Suppose $P_1, P_2,\ldots$ is a sequence of particles (or, equivalently, let $f_{P_1}, f_{P_2}, \dots$ be the sequence of associated conformal mappings) with $\diam(P_j)=d_j$. %We shall always assume that each particle
%can be generated by a driving function for the Loewner equation 
%(see Section \ref{Loewner}). 
Let $\theta_1, \theta_2, \ldots$ be a sequence of angles.
Define rotated copies of the maps $\{ f_{P_j}\}$ by setting
\[f_{P_j}^{\theta_j}(z)
=e^{i \theta_j}f_{P_j}(e^{-i \theta_j}z), \quad j=1,2, \ldots \,.\]
Take $\Phi_0(z)=z$, and recursively define
\begin{equation}
\Phi_n(z)=\Phi_{n-1}\circ f_{P_n}^{\theta_n}(z), \quad n=1,2,\ldots.
\label{HLrecipe}
\end{equation}
This generates a sequence of conformal maps $\Phi_n: D_0 \rightarrow D_n =\CC \setminus K_n$,
where $K_{n-1} \subset K_n$ are growing compact sets, which we usually call clusters. Loosely speaking we add, at the $n$th step, a particle of diameter 
$d_n|\Phi_{n-1}'(e^{i \theta_n})|$ to the previous cluster $K_{n-1}$ at the 
point $\Phi_{n-1}(e^{i\theta_n})$.

By constructing the sequences $\{\theta_j\}$ and $\{d_j\}$ in different ways, it is possible to describe a wide class of growth models. The most well-known are the Hastings-Levitov family of models ${\rm HL}(\alpha)$, indexed by a parameter $\alpha \in [0,2]$. Here the $\theta_j$ are chosen to be independent random variables distributed uniformly on the unit circle which, by conformal invariance,
corresponds to the attachment point being distributed according to 
harmonic measure at infinity. The particle diameters are taken as $d_j=d/|\Phi_{j-1}'(e^{i\theta_j})|^{\alpha/2}$. 

In this paper, we study a variant of the ${\rm HL}(0)$ model in which $\theta_1, \theta_2, \dots$ are independent identically distributed random variables on the unit circle $\TT$ with common law $\nu$ and $d_j=d$. We shall refer to this growth model as anisotropic Hastings-Levitov, ${\rm AHL}(\nu)$. Our limit results are not sensitive to the shapes of particles $P_j$ and, in fact, we are even able to relax the restraint $d_j=d$, to allow for $P_1, P_2, \dots $ to be chosen so that $d_1, d_2,\ldots$ are independent identically distributed random variables (independent of $\{\theta_j\}$) with law $\sigma$, satisfying certain conditions to be stated later. \footnote{One way of constructing such a sequence is by fixing a deterministic measurable mapping $d \mapsto f_{P(d)}$, such that $\diam(P(d))=d$, and then choosing an independent identically distributed sequence $d_1, d_2, \dots$ with law $\sigma$.} 

%The sequence $\{d_j\}$ represents
%diameters of particles $P_j$ (the actual shape of which we assume depends
%deterministically on $d_j$). 
%We usually assume that the shape of the
%particle is a deterministic function of $d$, so that $d \mapsto
%f_{P(d)}$ is deterministic.  
%We shall often assume that the $\theta_j$'s are iid, with common law 
%$\nu$. 
 
%A simple example of is provided by slit maps,
%\[f_{d_j}: D_0 \rightarrow D_0 \setminus (1,1+{d_j}],\]
%mapping the exterior disk onto the complement of the closed unit circle% with a slit of 
%random length $d_j$ attached.  At the $n$th step, we compose such a sli%t mapping with $\Phi_{n-1}$ to produce $\Phi_n$, and the slit gets mapp%ed onto part of a hyperbolic geodesic in $D_{n-1}$, of random length de%pending on $d_n$ and the shape of the cluster at the previous stage. 

\subsection{Background and motivation}
The motivation behind studying these clusters 
comes from growth processes that arise in physics, such as
diffusion-limited aggregation (DLA) \cite{WS}, anisotropic
diffusion-limited aggregation \cite{JuKoBo} and the Eden model \cite{Eden}.  In
1998, Hastings and Levitov \cite{HL} formulated a
conformal mapping approach to modelling Laplacian growth of which DLA
and the Eden model are special cases. They
defined the family of growth models, ${\rm HL}(\alpha)$, whose construction is described in the previous section.
The $\alpha=2$ version is a candidate for off-lattice DLA. In this case, the diameters of the mapped particles are (more or less) the same. 

The Hastings-Levitov model has been widely discussed in the physics
literature. In the original paper \cite{HL}, Hastings and Levitov
studied the model numerically and found evidence for a phase
transition in the growth behaviour at $\alpha=1$. Further numerical investigations can be seen in, for example, the papers \cite{David} and \cite{MS}.

Unfortunately, the Hastings-Levitov model has proved difficult to 
analyse rigourously, particularly in the $\alpha>0$ case. We give a 
brief review of the known results. In 2005, 
Rohde and Zinsmeister \cite{RZ} established the 
existence of limit clusters for $\alpha=0$ when the aggregate 
is scaled by capacity, and showed that the Hausdorff dimension of the
limit clusters is $1$, almost surely. They also considered a regularised
version of ${\rm HL}(\alpha)$ for $\alpha>0$ and estimated the growth 
rate of the capacity and length of the clusters. 
%In 2009, Johansson and Sola studied a variant of ${\rm HL}(0)$, 
%in which the particle sizes were random. This model arose from 
%Loewner chains driven by a certain L\'evy process.
In 2009, Johansson and Sola \cite{JS} studied
Loewner chains driven by compound Poisson processes. Certain cases 
of these were found to correspond to
${\rm HL}(0)$ clusters with random particle sizes, and 
the existence of (one-dimensional) limit clusters was established. 
The 2009 paper of Norris and Turner explored the evolution of harmonic 
measure on the boundary of ${\rm HL}(0)$ clusters and showed that this
converges to the coalescing Brownian flow.
We would finally like to mention the 2001 and 2002 papers of Carleson and Makarov 
(\cite{CM1}, \cite{CM2}), where the Loewner-Kufarev 
equation is used to describe deterministic 
versions of Laplacian growth.% (see also Makarov's 1999 paper \cite{Mak}). 

%We recover the original $\rm{HL}(0)$ model from the setup we have descr%ibed above by
%choosing the angles $\theta_j$ uniformly, and the diameters $d_j=d_0$ c%onstant.
%The diameters of the attached particles then vary 
%considerably since they are distorted by the conformal mapping --
%this %is in contrast to $\rm{HL}(\alpha)$ with $\alpha>0$, where the
%lengths %of the slits that are to be added are modified depending on
%the previou%s mapping $\Phi_{n-1}$ (see \cite{HL} for details). 

In this paper we have modified the setup of the Hastings-Levitov model in the $\alpha=0$ case.
The use of more general distributions for the angles is a way of introducing anisotropy or
localization in the growth. This is similar in spirit to the work of
Popescu, Hentschel, and Family \cite{PHF}, who study numerically 
a variant of ${\rm HL}(2)$, where the angles are distributed
according to a certain density with $m$-fold symmetry.
They suggest that such anisotropic
Hastings-Levitov models may provide a description for the growth of 
bacterial colonies where the concentration of nutrients is directional. We 
discuss their work further in the next section.

Allowing for non-uniform angular distributions
results in scaling limits in which the anisotropy is reflected.
We consider two different 
natural scaling limits where we scale the particle sizes.
We prove a shape theorem 
that describes the global macroscopic behaviour of the cluster: in 
the case of uniformly distributed angles,
the shape is a disk (as was previously known \cite{personalRohde}); but 
in the anisotropic case the limit shapes can be realised as
non-trivial Loewner
hulls. For the anisotropic case we also show that the 
evolution of harmonic measure on 
the cluster boundary is deterministic with small random fluctuations,
unlike in the uniform case where the behaviour is purely stochastic.
\subsection{Outline of the paper}
Our paper is organised as follows. In Section 2, we 
review some background material concerning the Loewner equation and the 
coalescing Brownian flow and describe the general framework of our paper.
We also discuss some examples of angular distributions that lead to 
interesting anisotropic behaviour in the growth. In Section 3, we establish 
continuity properties of the Loewner-Kufarev equation with respect to 
measures, and use this to prove a shape theorem for the limit clusters.
In Section 4, we consider the evolution of harmonic measure on the cluster 
boundary. For general measures, we 
first prove that the flow on the boundary is described 
by an deterministic ordinary differential equation, and then obtain a 
description of the stochastic fluctuations around this deterministic flow.
Finally, we show that uniformly chosen angles lead to purely stochastic 
behaviour, even if the particle sizes are chosen randomly.

\section{Preliminaries}
In this section we review some background material that is needed for 
our proofs.
\subsection{Loewner chains driven by measures}\label{Loewner}
%A general way to describe growing (random or deterministic) compact set%s is to
%use Loewner chains.
%Let $\mathcal{P}=\mathcal{P}(\TT)$ denote the class of probability 
%measures on $\TT$. We write $P$ for the collection of all analytic 
%functions in $\Delta$ 
%with positive real part and expansion $p(z)=1+\sum_{k=1}^{\iy}b_n z^{-n}$ 
%at infinity. By the Herglotz representation theorem 
%(see the book \cite{Duren}), 
%each measure $\mu \in \mathcal{P}$ corresponds to a function $p \in P$ via the integral
%\begin{equation*}
%p(z)=\int_{\mathbb{T}}\frac{z+\zeta}{z-\zeta}d\mu(\zeta), \quad z \in \Delta.
%\end{equation*}
A decreasing Loewner chain is a family of conformal mappings 
\begin{equation*}
f_t: D_0 \rightarrow \CC \setminus K_t, 
\quad \iy \mapsto \iy, \quad f_t'(\iy)>0,
\end{equation*} 
onto the complements of a growing family of compact sets, called hulls,
with 
\[
K_{t_1} \subset K_{t_2} \quad {\rm for} \quad t_1 < t_2.
\] 
We always take $K_0$ to be the closed unit disk. The logarithmic capacity 
of each $K_t$ is given by
\begin{equation*}
\lcap(K_t)=\lim_{z \to \iy}\frac{f_t(z)}{z}.
\end{equation*}
Let $\mathcal{P}=\mathcal{P}(\TT)$ denote the class of probability 
measures on $\TT$. 
Under some natural assumptions on the function $t \mapsto \lcap(K_t)$, 
such a chain can be parametrized in terms of 
%families $\{p_t \}_{t \geq 0}$ of functions in $P$.
families $\{\mu_t\}_{t\geq 0}$, $\mu_t \in \mathcal{P}(\TT)$. 
% For instance, one should require that 
%the function $t \mapsto \|\mu_t\|$ be measurable (see \cite{CM}).
More precisely, the conformal mappings $f_t$ satisfy the 
Loewner-Kufarev equation
%\begin{equation*}
%\partial_t f_t(z)=z f_t'(z) p_t(z)
%\label{LPDE}
%\end{equation*}
%with initial condition $f_0(z)=z$ (see \cite[chapter 3]{Duren}). 
%Rewriting this differential equation 
%in terms of measures in $\mathcal{P}$, we obtain
\begin{equation}
\partial_tf_t(z)=
z f_t'(z)\int_{\mathbb{T}}\frac{z+\zeta}{z-\zeta}\dx \mu_t(\zeta),
\label{measureLPDE}
\end{equation}
with initial condition $f_0(z)=z$. Conversely, if $t \mapsto
\|\mu_t\|$ is locally integrable (which is immediate for probability
measures) then the solution to \eqref{measureLPDE} exists and is a
Loewner chain. See \cite{CM1} for a general discussion.
%Here, one should require that 
%the function $t \mapsto \|\mu_t\|$ be measurable.

The classical example is the case of pure point masses 
\[\mu_t=\delta_{\xi(t)},\]
where $\xi=\xi(t)$ is a unimodular function. The Loewner-Kufarev equation then reduces to the equation
\begin{equation}
\partial_tf_t(z)=zf_t'(z)\frac{z+\xi(t)}{z-\xi(t)},
\label{classicalloewner}
\end{equation}
originally introduced by Loewner in 1923. The function $\xi
(t)$ is usually called the driving function. The particular choice
$\xi(t)=1$ produces as solutions the \emph{basic slit mappings} $f_{d(t)}: D_0 \to D_0
\setminus [1, 1+d(t)]$, with slit lengths $d(t)$  
given by the explicit formula
\begin{equation}
d(t)=2 e^{t}(1-\sqrt{1-e^{-t}})-2.
\label{slitformula}
\end{equation}
We can recover (the slit version of) the $\rm{HL}(0)$ mappings $\Phi_n$ by driving the Loewner equation with a non-constant point mass
at
\begin{equation} 
\xi(t)=\exp\left(i \sum_{j=1}^n\theta_j \chi_{[T_{j-1}, T_j]}(t)\right),
\label{HLmeasure}
\end{equation}
where the times $T_j$ relate to the slit lengths $d$ via the formula
(\ref{slitformula}).

Choosing absolutely continuous driving measures 
\[\dx \mu_t= h_t(\zeta) |\dx \zeta|\]
results in the growth of the clusters no longer being concentrated at a 
single point. In the simplest 
case $\dx \mu_t(\zeta)=|\dx \zeta|/2\pi$,
the Loewner-Kufarev equation reduces to
\[\partial_t f_t(z)=z f_t'(z),\]
and we see that $f_t(z)=e^tz$, so that $K_t=e^t K_0$. We shall see that absolutely
continuous driving measures arise naturally in connection with the anisotropic ${\rm HL}(0)$
clusters.

We can realize more general particles than slits using a driving function in the
following way. Consider a particle $P$ such that $\partial P \cap D_0$ 
can be described by a (sufficiently smooth) crosscut $\beta$ of $D_0$. 
We parametrize $\beta(t)$ according to capacity, that is,
$\lcap(K_0 \cup \beta[0,t))=e^t, \, t \in [0,T_P)$, where
\[T_P=\ellcap(P):=\log(\lcap (K_0 \cup P)).\] 
We can then find a driving function
for the Loewner equation that produces a family $f_t:D_0 \to D_0 \setminus \beta[0,t)$. As $t \to T_P$, the
conformal maps $f_t$ converge uniformly on compact subsets of $D_0$ to the
mapping $f_P: D_0 \to D_0 \setminus P$. If $\xi: [0, T_P) \rightarrow \TT$
denotes the driving function for a single particle, we 
obtain a driving function for the cluster similarly to
\eqref{HLmeasure}. 

There is a useful relation between the diameter of
the particle, its capacity, and the driving function. Set
$R(P):=\sqrt{T_P}+\sup_{0 \le t \le T_P}|\xi(t)|$. Then, as is proved
in \cite[Lemma 2.1]{LSW}, we have 
\begin{equation}\label{capvsdiam}
d(P)=\diam(P) \asymp R(P).
\end{equation}
%\begin{equation*}
%\xi(t)=\sum_{j=1}^{\iy}\xi_j(t)
%\end{equation*}
%with $\xi_j(t)=e^{i \theta_j}\xi(t-T_P(j-1))$ for $t \in [T_P(j-1),T_P j)$, 
%with the $\theta_j$:s again chosen according to $\nu$.
Moreover, one can prove that there exists a constant $c < \infty$ such that
\[
c^{-1} h^2 \le \ellcap(P) \le c\, hd
\]
for small $d,h$, where $h=\sup\{|z|: z \in P\}-1$. Indeed, the first inequality follows by comparing with the slit map solution to the Loewner equation. The second follows from a harmonic measure estimate and the identity $\ellcap(P)=\EE[\log|B_{\tau}|]$, where $B_t$ is a planar Brownian motion starded from $\infty$ and $\tau$ is the hitting time of $K_0 \cup P$. In particular we see that there are natural sequences of particles such that $\ellcap(P) \asymp d^2$ as $d \to 0$. We shall make this assumption in certain sections of this paper.
 
We will sometimes need to consider the radial Loewner equation lifted to the real line: let $\gamma_t(x)=-i\log g_t(e^{2\pi i x})/2\pi$, where $g_t=f_t^{-1}$ and 
$f_t$ is a solution to \eqref{measureLPDE} and $x \in \RR$. Then $\gamma_t$ satisfies the 
differential equation
\begin{equation}
\partial_t\gamma_t(x)=
\frac{1}{2\pi}\int_{0}^1\cot\left(\pi (\gamma_t(x)-y)\right)\dx \mu_t(e^{2\pi iy}),
\label{Hilbert}
\end{equation}
with $\gamma_0(x)=x$ (see \cite[Chapter 4]{Lawler}). This is well-defined 
as long as $\gamma_t(x)$ is outside the support of $\mu_t$. However, we may
interpret the integral in the sense of principal values, that is, as a multiple
of the Hilbert transform of the measure $\mu_t$ (see \cite[Chapter 3]{Garnett}),
\[H[\nu_t](x)={\rm p.v.}\,\frac{1}{2\pi}
\int_0^1\cot(\pi(x-y))\dx \mu_t(e^{2\pi iy}).\]
In this way, for nice enough measures, we 
obtain a differential equation defining a flow on all of $\TT$.

\subsection{Coalescing Brownian flow and harmonic measure on the cluster boundary}\label{BrownianWeb}
The coalescing Brownian flow (also known as the Arratia flow and the Brownian web) can loosely be defined as a family of coalescing Brownian motions, starting at all possible points in continuous space-time. Arratia \cite{A79} first
considered this object in 1979 as a limit for discrete coalescing random walks. Since then it has been studied by, amongst others, T{\'o}th and Werner \cite{TW}, Fontes, Isopi, Newman
and Ravishankar \cite{FINR} and recently Norris and Turner \cite{NT}. One of the difficulties in studying the coalescing Brownian flow is constructing a suitable topological space on which a unique measure with the necessary properties exists. In this section we outline the construction of Norris and Turner \cite{NT} and show how the coalescing Brownian flow relates to the evolution of harmonic measure on the boundary of the ${\rm AHL}(\nu)$ clusters.

Let $\mathcal{R}$ be the set of non-decreasing, right-continuous functions $f^+:\RR\to\RR$ that satisfy the property
$$
f^+(x+n) = f^+(x) + n, \quad x\in\RR,\quad n \in \ZZ.
$$
Write $\mathcal{L}$ for the analogous set of left-continuous functions and let $\mathcal{D}$ be the set of all pairs $f=\{f^-,f^+\}$, where $f^-$ is the left-continuous modification of $f^+$.
Since $x+f^+(x)$ is strictly increasing in $x$, there is for each $t\in\RR$
a unique $x\in\RR$ such that
$$
\frac{x+f^-(x)}2\leq t\leq\frac{x+f^+(x)}2.
$$
Let $f^\times(t)=t-x$. A metric $d_{\mathcal{D}}$ is defined on $\mathcal{D}$ by
$$
d_\mathcal{D}(f,g)=\sup_{t\in[0,1)}|f^\times(t)-g^\times(t)|.
$$
Under this metric, the space $(\mathcal{D}, d_{\mathcal{D}})$ is complete and separable. Let $\mathcal{D}_0$ be the set of circle maps whose liftings are in $\mathcal{D}$.

%The space of continuous weak flows  $C^\circ((0,\infty),\mathcal{D}_0)$ consists of flows %$\phi=(\phi_{ts}:s,t\in(0,\infty),s<t)$, with $\phi_{ts}\in\mathcal{D}_0$ for all $s,t$, such that
%$$
%\phi_{ut}^-\circ\phi_{ts}^-\leq\phi_{us}^-\leq\phi_{us}^+\leq\phi_{ut}^+\circ\phi_{ts}^+,
%\quad s<t<u
%$$
%and, for all $t\in (0,\infty)$,
%$$
%\phi_{ts}\to\operatorname{id}\text{ as $s\uparrow t$},\quad\phi_{ut}\to\operatorname{id}\text{ as $u\rightarrow t$}.
%$$
%Define, for $\phi,\psi\in C^\circ((0,\infty),\mathcal{D}_0)$,
%$$
%d_C(\phi,\psi)=\sum_{n=1}^\infty2^{-n}\{d_C^{(n)}(\phi,\psi)\wedge1\},
%$$
%where
%$$
%d_C^{(n)}(\phi,\psi)=\sup_{s,t\leq n,s<t}d_{\mathcal{D}}(\phi_{ts},\psi_{ts}).
%$$
%Then $d_C$ is a metric on $C^\circ((0,\infty),\mathcal{D}_0)$, under which this space is
%complete and separable.

Write $I=I_1\oplus I_2$ if $I_1, I_2$ are disjoint intervals with $\sup I_1=\inf I_2$ and $I=I_1\cup I_2$. 
The set of cadlag weak flows $D^\circ$ consists of flows $\phi=(\phi_I:I\subseteq [0,\infty))$, where $\phi_I\in\mathcal{D}_0$ and $I$ ranges
over all non-empty finite intervals that satisfy
$$
\phi_{I_2}^-\circ\phi_{I_1}^-\leq\phi_I^-\leq\phi_I^+\leq\phi_{I_2}^+\circ\phi_{I_1}^+,
\quad I=I_1\oplus I_2
$$
and, for all $t\in (0,\infty)$,
$$
\phi_{(s,t)}\to\operatorname{id}\quad\text{as $s\uparrow t$},\quad
\phi_{(t,u)}\to\operatorname{id}\quad\text{as $u\downarrow t$}.
$$
For $\phi,\psi\in D^\circ$ and $n\geq1$, define
$$
d_D^{(n)}(\phi, \psi) = \inf_\lambda\left\{
\gamma(\lambda) \vee \sup_{I\subset (0,\infty)}\|\chi_n(I)\phi_I^\times-\chi_n(\lambda(I))\psi_{\lambda(I)}^\times\|\right\},
$$
where the infimum is taken over the set of increasing homeomorphisms
$\lambda$ of $(0,\infty)$, where
$$
\gamma(\lambda) = \sup_{t\in(0,\infty)}|\lambda(t)-t|\vee\sup_{s<t}\, \left|\log\left(\frac{\lambda(t) - \lambda(s)}{t-s}\right)\right|,
$$
and where $\chi_n$ is the cutoff function given by
$$
\chi_n(I)=0\vee(n+1-\sup I)\wedge1.
$$
Define
$$
d_D(\phi, \psi)=\sum_{n=1}^\infty2^{-n}\{d_D^{(n)}(\phi,\psi)\wedge1\}.
$$
Then $d_D$ is a metric on $D^\circ$ under which $D^\circ$ is
complete and separable. 
%Moreover, the metrics $d_C$ and $d_D$ generate the same topology on $C^\circ((0,\infty),\mathcal{D}_0)$.

For $e=(s,x) \in [0,\infty) \times \RR$, let $D_e=D_x([s,\infty),\RR)$ denote the Skorokhod space of cadlag paths starting from $x$ at time $s$. Write $\mu_e$ for the distribution on $D_e$ of a standard Brownian motion starting from $e$. For a sequence $E=(e_k:k\in\NN)$ in $[0,\infty) \times\RR$, where $e_k=(s_k,x_k)$ say, let $D_E=\prod_{k=1}^\infty D_{e_k}$, be the complete separable metric space with metric $d_E$ on $D_E$ defined by
$$
d_E(\xi,\xi')=\sum_{k=1}^\infty 2^{-k}\{d(\xi^k,{\xi'}^k)\wedge 1\},
$$
where $d$ denotes appropriate instances of the Skorokhod metric. There exists a unique probability measure $\mu_E$ on $D_E$ under which the coordinate processes on $D_E$ are coalescing Brownian motions. Define a measurable map $Z^{e,+}:D^\circ \to D_e$ by setting
$$
Z^{e,+}(\phi)=(\phi_{(s,t]}^+(x):t\geq s),
$$
and a measurable map
$Z^{E,+}:D^\circ \to D_E$ by
$$
Z^{E,+}(\phi)^{e_k}=Z^{e_k,+}(\phi).
$$
There exists a unique Borel probability measure $\mu_A$ on $D^\circ$
such that, for any finite set $F\subset [0,\infty) \times \RR$, we have
$$
\mu_A\circ (Z^{F,+})^{-1}=\mu_F.
$$
We call any $D^\circ$-valued random variable with law $\mu_A$ a coalescing Brownian flow on the circle (see Figure 1).
\begin{figure}
\label{BW_fig}
\centering
%change picture file names to .eps!
\includegraphics[width=0.45 \textwidth]{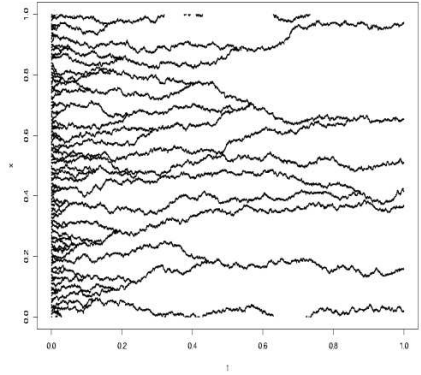}
\caption{A realisation of the coalescing Brownian flow, with only paths starting at time 0 shown.} 
\end{figure}

Recall the construction of the Hastings-Levitov clusters from the Introduction. Let $P$ be a closed, connected, simply connected subset of $D_0$ with $P \cap K_0 = \{1\}$. Write $g_{P}$ for the inverse mapping from $D_1 \to D_0$. %This map
%extends continuously to the boundary of $\partial D_1$. 
There exists a unique 
$\gamma_P \in \mathcal{D}$ such that $\gamma_P$ restricts to a continuous map from the interval $(0,1)$ to itself, and such that
\[
g_P(e^{2\pi ix})=e^{2\pi i\gamma_P(x)},\quad x\in(0,1).
\]
Set $\Gamma_n=g_{P_n}\circ\dots\circ g_{P_1}$, where $g_{P_n}=(f_{P_n}^{\theta_n})^{-1}$, so that $\Gamma_n: D_n \to D_0$. 
The extension of $\Gamma_n$ to the boundary $\partial K_n=\partial D_n$, gives a natural parametrization of the boundary of the $n$th cluster by the unit circle. It has the property that, for 
$\xi,\eta\in\partial K_n$,
the normalized harmonic measure $\omega$ (from $\infty$) of the positively oriented boundary segment from $\xi$ to $\eta$
is given by $\Gamma_n(\eta)/\Gamma_n(\xi)=e^{2\pi i\omega}$.
For $m,n \in \NN$ with $m<n$, set
$$
\Gamma_{nm}=g_{P_n}\circ\dots\circ g_{P_{m+1}}|_{\partial K_0}.
$$
Set $\Gamma_{nn}=\operatorname{id}$. The circle maps $\Gamma_{nm}$ have the flow property
$$
\Gamma_{nm}\circ \Gamma_{mk}=\Gamma_{nk},\quad k\leq m\leq n.
$$
The map $\Gamma_{nm}$ expresses how the harmonic measure on $\partial K_m$ is transformed by the arrival of new particles
up to time $n$. Suppose $0<T_1 <T_2 < \cdots$ are times of a Poisson process, independent of $\{\theta_j\}$ and $\{d_j\}$, with rate to be specified later. Embed $\Gamma$ in continuous time by defining, for an interval $I\subseteq [0,\infty)$, $\Gamma_I=\Gamma_{nm}$ where $m$ and $n$ are the smallest and largest integers $i$, respectively, for which $T_i \in I$. Then $(\Gamma_I:I\subseteq [0,\infty))$ is a random variable in $D^\circ$. We denote its law by $\mu^P_A$.

In the paper \cite{NT}, Norris and Turner showed that for ${\rm HL}(0)$ clusters (i.e. clusters where the particles $P_j$ have constant diameters $d_j=d$, and $\theta_j$ is uniformly distributed on the circle), in the case of symmetric particles, $\mu_A^P\to\mu_A$ weakly on $D^\circ$ as $d\to0$, where the Poisson process $\{T_i\}$ has rate $\rho(P) \asymp d^{-3}$, defined by
$$
\rho(P) \int_0^1 (\gamma_P(x)-x)^2 \dx x = 1.
$$
If $P$ is not symmetric, the same result holds once the definition of $\Gamma_I$ is modified to 
$$
\Gamma_I(e^{2\pi ix})=e^{-2\pi i\beta t}\Gamma_{nm}(e^{2\pi i(x+\beta s)}),
$$
where $s=\inf I$ and $t=\sup I$ and $\beta = \beta(P)$ is defined by 
$$
\beta(P) = \rho(P) \int_0^1 (\gamma_P(x)-x) \dx x = \mathcal{O}(d^{-1}).
$$
In other words the following result about the evolution of harmonic measure on the cluster boundary holds.

Let $x_1,\dots,x_n$ be a positively oriented set of points in $\RR/\ZZ$ and set $x_0=x_n$.
Set $K_t=K_{\lfloor\rho(P)t\rfloor}$. For $k=1,\dots,n$, write $\omega_t^k$ for the harmonic measure in $K_t$ of the boundary segment of all fingers in $K_t$ attached between $x_{k-1}$ and $x_k$.
Let $(B_t^1,\dots,B_t^n)_{t\geq0}$ be a family of coalescing Brownian motions in $\RR/\ZZ$ starting from $(x_1,\dots,x_n)$. Then, in the limit $d\to0$, $(\omega_t^1,\dots,\omega_t^n)_{t\geq0}$ converges weakly in $D([0,\infty),[0,1]^n)$ to $(B_t^1-B_t^0,\dots,B^n_t-B_t^{n-1})_{t\geq0}$.

In this paper, we extend the study of Norris and Turner to cover random arrival points $\theta_j$ with law $\nu$, as well as random particle diameters $d_j$ with law $\sigma$. As we shall see, the results of \cite{NT} go through more or less unchanged in the case when $\nu$ is the uniform measure, for all laws $\sigma$ with finite third moments that tend to zero. However, in the case of non-uniform measures $\nu$, the evolution of harmonic measure on the boundary is dominated by a non-trivial deterministic drift of order $d^{2}$, and the stochastic behaviour is seen only as fluctuations about this of order $d^3$.

\subsection{Some examples}
We give two examples of anisotropic growth to illustrate our results. We consider the case 
of slit mappings with deterministic length $d$ for convenience.

\subsubsection{Angles chosen in an interval}
\begin{figure}[h]
%\vspace{-2pt}
    \subfigure[${\rm AHL}(\nu)$ cluster (left) and the corresponding Loewner hull (right).]
        {\epsfig{file=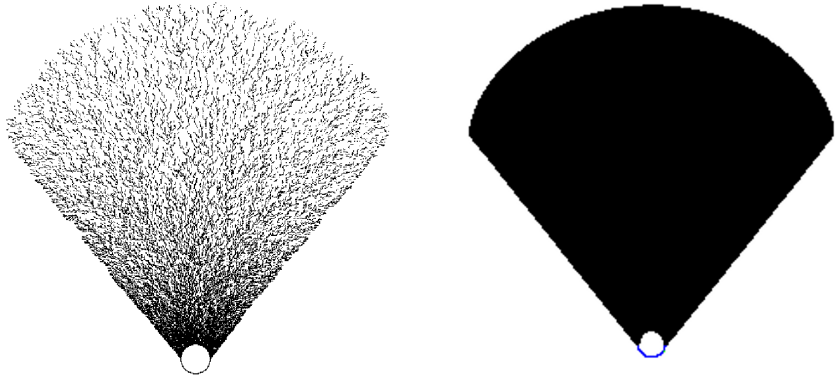,width=12cm}}
   \subfigure[Evolution of harmonic measure on the boundary of ${\rm AHL}(\nu)$ (left) and the solution to the corresponding deterministic ODE (right).]
        {\epsfig{file=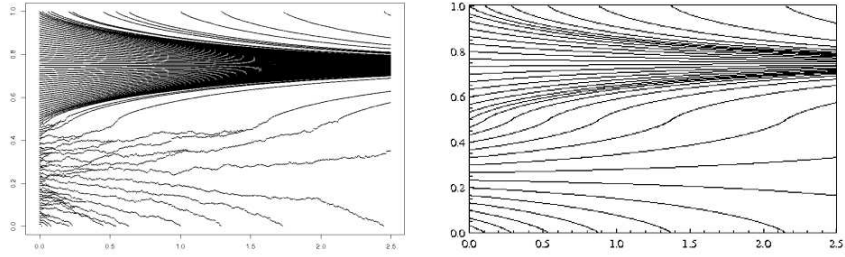,width=12cm}}
  \caption{\textsl{Simulations of ${\rm AHL}(\nu)$ and associated limits, for $d = 0.02$ after $25000$ repetitions, corresponding to $\dx \nu(x) = 2 \chi_{[0,1/2]}\dx x$.}}
\end{figure}

For $\eta \in (0, 1]$, let $\theta_j$ be chosen 
uniformly in $[0, \eta]$. We build clusters $K_n$ as before, at each step 
setting  $d_j=d$ for $j=1,\ldots,n$. For fixed $t \in (0, \iy)$, if $n = \lfloor \ellcap(P)^{-1} t \rfloor$, the hull $K_n$ produced by the discrete 
iteration model then converges (in a sense to be made precise) 
as $d \rightarrow 0$, to the hulls obtained by solving the Loewner equation at time $t$ 
driven by the measure 
\begin{equation*}
\dx \nu(e^{2\pi ix})=\frac{\chi_{[0,\eta]}(x)\dx x}{\eta}.
\end{equation*}
A computation involving the power series 
expansion of the Schwarz-Herglotz kernel $(z+e^{2\pi i x})/(z-e^{2\pi i x})$
shows that the explicit form of the Loewner equation in this case is
\begin{equation}
\partial_t f_t(z)=z f_t'(z)\left(1+\frac{2}{\eta}
\arctan \left[\frac{e^{i \pi \eta}\sin \eta}{z-e^{i \pi \eta}\cos \eta}\right]\right).
\end{equation}
Construct the flow $\Gamma \in D^\circ$ that describes the evolution of harmonic measure on the cluster boundary, with rate $\ellcap(P)^{-1} \asymp d^{-2}$. Then, as $d \rightarrow 0$, $\Gamma \rightarrow \phi$ in $(D^\circ, d_D)$, where $\phi_{(s,t]}(x)$ is the solution to the ordinary differential equation
\[
\dot{\phi}_{(s,t]}(x)=\frac{1}{\pi^2} \log \left|\frac{\sin(\pi \phi_{[s,t)}(x))}
{\sin(\pi (\phi_{(s,t]}(x)-\eta))}\right|
\]
with $\phi_{(s,s]}(x)=x$. In the special case $\eta=1/2$, we obtain the equation
\[
\dot{\phi}_{(s,t]}=\frac{1}{\pi^2} \log |\tan (\pi \phi_{(s,t]}(x))|.
\]
Note the absence of random fluctuations in the region $(1/2,1)$ in the 
simulation in the figure; this phenomenon will be discussed in Section 4.
\subsubsection{Angles chosen from a density with $m$-fold symmetry}

\begin{figure}[h]
%\vspace{-2pt}
    \subfigure[${\rm AHL}(\nu)$ cluster (left) and the corresponding Loewner hull (right).]
        {\epsfig{file=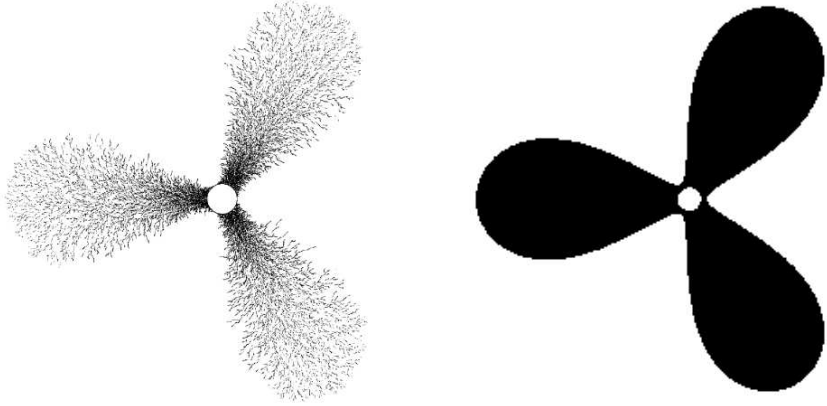,width=12cm}}
   \subfigure[Evolution of harmonic measure on the boundary of ${\rm AHL}(\nu)$ (left) and the solution to the corresponding deterministic ODE (right).]
        {\epsfig{file=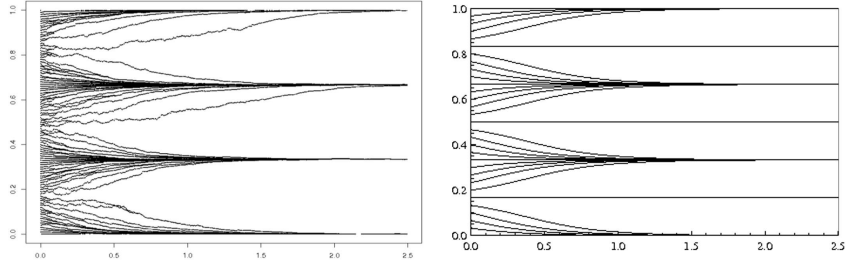,width=12cm}}
  \caption{\textsl{Simulations of ${\rm AHL}(\nu)$ and associated limits, for $d = 0.02$ after $25000$ repetitions, corresponding to $\dx \nu(x)=2\sin^2(3\pi x)\dx x$.}}
\end{figure}

For fixed $m \in \NN$, choose
$\theta_j$ distributed according to the density
\[\dx \nu(e^{2\pi ix})=2\sin^2(m\pi x)\dx x.\]
This type of density with $m$-fold symmetry is 
considered in \cite{PHF} as an example of 
a choice of angular distribution that introduces certain preferred 
directions in the cluster growth. The clusters converge, under the same scaling limits as above, to 
the hulls of the Loewner chain described by the equation
%\begin{equation}
%\partial_t f_t(z)=z f_t'(z)\left(1+\frac{2i}{\pi}
%\log \left[\frac{z}{z-1}\right]\right).
%\end{equation}
\begin{equation*}
\partial_t f_t(z)=z f_t'(z)\left(1-\frac{1}{z^{m}}\right).
\end{equation*}
In the limit in this case, 
the evolution of harmonic measure on the cluster boundary
is determined by the solutions to the ODE
\begin{equation*}
\dot{\phi}_{(s,t]}(x)=-\frac{1}{2\pi}\sin(2\pi m \phi_{(s,t]}(x)), \quad \phi_{(s,s]}(x)=x.
\end{equation*}
\section{A shape theorem}
In this section, we consider a scaling limit where the particle
sizes converge to zero. The goal is to describe the macroscopic shape of
the limiting cluster, that is, to prove a shape theorem.
This generalizes a result we first learned about from Rohde
\cite{personalRohde}, see also \cite{MS}: simulations of standard $\rm{HL}(0)$ clusters show that if the basic slit length $d$ is chosen to be small,
and the number of compositions is large, then the clusters $K_n$ look rounded. 
In fact, if we let $d \rightarrow 0$ and $n \asymp d^{-2}$, then the laws of resulting $\rm{HL}(0)$ clusters do indeed 
converge to that of a closed disk $c K_0$.

Similarly, comparing the ${\rm AHL}(\nu)$ clusters with the hulls generated by the Loewner equation driven 
by the time-independent measure $\nu$, we see that, as the particle diameters 
$d$ tend to zero and the number of compositions increases at a rate
proportional to $d^{-2}$, the shapes converge (even for random
particle sizes). Indeed, in Theorem \ref{shape} we prove that the discrete 
clusters converge to the Loewner hulls. We begin
with a technical result about solutions to the Loewner-Kufarev equation.
\subsection{Continuity properties of the Loewner equation}
%Consider families $\{\mu_t\}_{t \geq 0}$ of measures 
%with $\mu_t \in \mathcal{P}$ for every $t \geq 0$.  
In this section, we show that solutions to the Loewner-Kufarev 
equation (\ref{measureLPDE}) are ``close'' at time $T$ if the driving 
measures are ``close'' in some suitable sense. For conformal mappings, the 
notion of closeness is to be understood in the sense of uniform convergence on 
compact subsets of $D_0$. 

Let $\Sigma$ denote the space of conformal mappings $f: D_0 \rightarrow \CC$ 
with expansions at infinity of the form
\[f(z)=c_1z+c_0+c/z+\cdots, \quad c_1>0,\]
equipped with the topology induced by uniform convergence on 
compact subsets of $D_0$. Denote by $\Pi(\Sigma)$ 
the space of probability measures on $\Sigma$.

In \cite{Bauer}, it is shown that if the Loewner equation is driven by continuous functions that are 
close in the uniform metric, then the corresponding solutions are close as conformal mappings. This was extended to cover 
Skorokhod space functions in \cite{JS}.

The following proposition deals with the case of general driving measures. 
\begin{Pro}
\label{continuity}
Let $0<T< \iy$. Let $\mu^n=\{\mu^n_t\}_{t \geq
  0}, \, n=1,2,\ldots,$ and $\mu=\{\mu_t\}_{t \geq
  0}$ be families of measures 
in $\mathcal{P}$. Let
$m$ denote Lebesgue measure on $[0, \iy)$, and suppose that
the measures $\mu^n_t \times m$ converge weakly on $S=\TT\times [0, T]$
to the measure $\mu_t \times m$ as $n \rightarrow \iy$. 

Then the solutions 
$\{f^n_T\}$ to (\ref{measureLPDE}) corresponding to the sequence $\{\mu^n\}$
converge to $f_T$, the solution corresponding to 
$\mu$, uniformly on compact subsets of $D_0$.
\end{Pro}
\begin{proof}
The proof is similar to the continuity lemmas of \cite{JS} 
and \cite{Bauer}.

Fix a compact set
$K \subset  D_0$, and let $\epsilon>0$ be given. For $t \in [0, T]$, 
consider the backward Loewner flow 
\begin{equation}
\dot{h}_t(z)=-h_t(z)\int_{\TT}\frac{h_t(z)+\zeta}
{h_t(z)-\zeta}\dx \mu_{T-t}(\zeta),\quad h_0(z)=z.\label{charloewner}
\end{equation}
We then have $h_T(z)=f_T(z)$ (see \cite[Chapter 4]{Lawler}); and a similar statement holds 
for the solutions corresponding to the measures $\mu^n$. We shall
use the convenient shorthand notations
\begin{equation}
v(s,\nu, z)=-z\int_{\TT}\frac{z+\zeta}{z-\zeta}\dx \nu_{T-s}(\zeta)
\label{vshort}
\end{equation}
and 
\begin{equation}
w(x,z)=z\frac{z+x}{z-x}.
\label{wshort}
\end{equation}

For $z \in K$ fixed, set $u(t)=h_t(z)$ and $u^n(t)=h^n_t(z)$. Integrating 
(\ref{charloewner}) with respect to $t$ and using 
$u(0)=u^n(0)$, we obtain, for $t \in [0,T]$,
\begin{align*}
\left|u(t)-u^n(t)\right|& \leq \left|u(t)-u(0)-\int_0^tv(s,\mu^n,u(s))\dx s\right|\\
& +\left|\int_0^t v(s,\mu^n,u(s))\dx s-u^n(t)+u^n(0)\right|\\
& =\left|\int_0^tv(s,\mu,u(s))\dx s-\int_0^tv(s,\mu^n,u(s))\dx s\right|\\&+
\left|\int_0^tv(s,\mu^n,u(s))\dx s-\int_0^tv(s,\mu^n,u^n(s))\dx s\right|.
\end{align*}
The first term may be written out as
\begin{align*}
&\left|\int_0^tv(s,\mu,u(s))\dx s-\int_0^tv(s,\mu^n,u(s))\dx s\right| \\
=&\left|\int_0^t\int_{\TT}u(s)\frac{u(s)+x}{u(s)-x}\dx \mu_{T-s}(x)\dx s 
 -\int_0^t\int_{\TT}u(s)\frac{u(s)+x}{u(s)-x}\dx \mu^n_{T-s}(x)\dx s\right|,
\end{align*}
and since 
the integrand is a continuous function on $\TT \times [0,T]$, 
our assumption of weak convergence implies, in particular, that 
the right-hand side is smaller than $\epsilon$ for all $n\geq N$ for some $N$ 
(which depends on the point $z$).

We estimate the second term by
\begin{multline*}
\left|\int_0^tv(s,\mu^n,u(s))\dx s-\int_0^tv(s,\mu^n,u^n(s))\dx s\right|\\
\leq \int_0^t\int_{\TT}|w(x,u(s))-w(x,u^n(s))|\dx \mu^n_{T-s}(x)\dx s.
\end{multline*}
We now use the inequality
\begin{equation*}
|w(x,z)-w(x,z')|\leq (\sup|\partial_z w(x,z)|)|z-z'|
\end{equation*}
together with standard growth estimates on 
conformal mappings of $D_0$ to obtain that
\begin{equation*}
|w(x,u(s))-w(x,u_n(s))|\leq C(T,K)|u(s)-u^n(s)|
\end{equation*}
for some constant $C(T,K)$ that does not depend on $n$ 
(the lack of normalization of the mappings accounts for the dependence on 
$T$). This in turn leads to the estimate
\begin{align*}
\int_0^t\int_{\TT}|w(x,u(s))-w(x,u^n(s))|\dx \mu^n_{T-s}(x)\dx s &\\
\leq C(T,K)\int_0^t \|\mu^n_{T-s}\| |u(s)-u^n(s)|\dx s&\\
= C(T,K)\int_0^t|u(s)-u^n(s)|\dx s.&
\end{align*}

Putting everything together, we find that
\begin{equation*}
|u(t)-u^n(t)|\leq \epsilon+C(T,K)\int_0^t|u(s)-u^n(s)|\dx s.
\end{equation*}
We are now in a position to apply Gr\"onwall's lemma, and we obtain
\begin{equation*}
|u(t)-u^n(t)|\leq C'(T,K) \epsilon, \quad t \in [0, T].
\end{equation*}
Thus $|u(T)-u^n(T)|<C'(T,K)\epsilon$
for $n\geq N$, and this means that $f^n_T(z)$ converges to $f_T(z)$.

We have thus established the pointwise convergence of $\{f_n\}$ to $f$
on the compact set $K$. 
Since the sequence $\{f_n\}$ is locally bounded by distortion-type 
estimates, it follows from Vitali's theorem that the convergence is in fact 
uniform on $K$, and the proof is complete.
%Now, since the sequence $\{f_n\}$ is uniformly bounded on compacts, it 
%belongs to a normal family which also contains the limit function $f$. 
%Hence, the pointwise convergence of the functions $f_n$ guarantees the 
%uniform convergence on compacts, and the proof is complete.
\end{proof}
Let $\mathbb{T}$ be the unit circle and set $S=\mathbb{T}
\times [0, \iy)$. Let $\mathcal{M}=\mathcal{M}(S)$ be the set of 
locally bounded Borel measures on $S$. A sequence $\mu, \mu^n \in
\mathcal{M}$ is said to converge vaguely if 
\[\int \varphi \dx \mu^n \to \int \varphi
\dx \mu, \quad \forall \varphi \in \mathcal{C}_c(S),
\] 
where $\mathcal{C}_c(S)$ is the set of continuous functions in
$S$ with compact support. Weak convergence is defined the
same way with compactly supported continuous $\varphi$ replaced by bounded
continuous $\varphi$. 
%equip $\mathcal{M}$ with the topology induced by vague
%convergence, and the corresponding Borel $\sigma$-algebra. 
A random measure on $S$ is a measurable mapping from some probability space 
into
$\mathcal{M}$. %A sequence of random measures $\mu^n$ on $S$ is said to
%converge in distribution if the corresponding probability laws
%converge weakly. 
In the next section we shall need the following lemma contained in
\cite[Theorem 15.7.6]{Kallenbergmeas}.
\begin{lemma}\label{convergence_in_dist}
Let $\mu,\mu^n$, $n=1,2,\ldots$, be bounded measures on
$S$. The sequence $\mu^n \rightarrow \mu$ with respect to the weak topology
as $n \to \iy$ if and only if
%\begin{equation}\label{vague_conv}
%\int \varphi \dx \mu^n \stackrel{d}{\to} \int \varphi \dx \mu, \quad \forall \varphi \in \mathcal{C}_c(S)
%\end{equation}
$\mu^n \rightarrow \mu$ in the vague topology
and $\mu^n(S)\rightarrow \mu(S)$.
%\begin{equation}\label{vague_conv2}
%\mu^n(S) \stackrel{d}{\to} \mu(S),
%\end{equation}
\end{lemma}
%where $\stackrel{d}{\to}$ means convergence in distribution.
%Note that (\ref{vague_conv}) is equivalent to convergence in
%distribution with respect to the vague topology, although in this case one does
%not have to require the measures to be bounded. 
%\begin{remark}\label{convergence_as}
%Note that if we replace the distributional convergence in \eqref{vague_conv} and \eqref{vague_conv2} by almost sure convergece, we may conclude that the measures $\mu^n$ converge to $\mu$ almost surely with respect to the weak topology.    
%\end{remark}
\subsection{Statement and proof of the shape theorem}
In this section, we prove the convergence of the random measures 
generating $\rm{AHL}(\nu)$ to the desired deterministic
measure when the particle diameter $d$ tends to zero and the
number of compositions tends to infinity at a rate proportional to $d^{-2}$. In view of the continuity result of 
the previous section, the weak convergence of the ${\rm AHL}(\nu)$
mappings then follows.

Let $P_1, P_2, \dots$ be chosen to be identical with $\diam(P)=d$. Assume additionally that the particle shape is chosen with capacity $\ellcap(P)$ of order $d^2$ (see \eqref{capvsdiam}). Note that our results can be shown to hold when the $P_j$ are random, under additional conditions that are given in the remark at the end of this subsection. Let $\theta_1, \theta_2, \ldots$ be $\mathbb{T}$-valued random variables with law $\nu$.

\begin{theorem}\label{shape}
Let $\Phi$ denote the solution to the Loewner-Kufarev equation driven by 
the measures $\{\nu_t\}_{t\geq 0}=\{\nu\}_{t\geq 0}$ and 
evaluated at time $T$, for some fixed $T \in (0, \infty)$.

Set $n = \lfloor \ellcap(P)^{-1} T \rfloor$, and define the conformal map
\begin{equation*}
\Phi_n=f_{P_1}^{\theta_1}\circ \cdots \circ f_{P_n}^{\theta_n}.
\end{equation*}
Then $\Phi_n$ converges to $\Phi$ uniformly on compacts almost surely 
as $d \rightarrow 0$. 
\end{theorem}
% \[
% \Xi_n(t)=
% \sum_{k=1}^{n}\theta_k\chi_{\{T_{k-1}^n \le t < T_{k}^n\}}(t).
% \]

\begin{proof}
Let $\epsilon>0$ be given. For $k=1, \ldots, n$, set
\[
T_k=k\ellcap(P), 
\] 
and
\[
 \Xi_n(t)=
 \sum_{k=1}^{n}\chi_{[T_{k-1}, T_k)}(t)\,\xi_k(t),
\]
where $\xi_k(t), \, t\in [T_{k-1}, T_k)$, is the (rotated)
driving function for the particle $P_k$. 
We set $\xi^n(t)=\exp(i \Xi_n(t))$.
Then $\delta_{\xi^n(t)}$ is the measure that drives the evolution of 
the $\rm{AHL}$ clusters. That is, the mapping 
$\Phi_n$ is the solution to the Loewner-Kufarev equation
\begin{equation}
\partial_t f_t(z)=zf_t'(z)\int_{\TT}\frac{z+\zeta}{z-\zeta}
\delta_{\xi^n(t)}(\zeta)
\label{deltaloewner}
\end{equation}
with $f_0(z)=z$, evaluated at time $t=T_n$.
Integrating with respect to Lebesgue measure in time, $m$, we see that
we need to show that the random
measures 
\begin{equation*}
\mu^P=\delta_{\xi^n(t)} \times m_{[0,T_n]} \in \mathcal{M}(S)
\end{equation*}
converge almost surely to 
$\mu=\nu\times m_{[0,T]}$ as $d \rightarrow 0$ with respect
to the weak topology. Note that
$\mu^P(S)=T_n \to T= \mu(S)$. By Lemma \ref{convergence_in_dist} it remains to prove convergence of the random variables
 $\langle \mu^P, \varphi \rangle$ to $\langle \mu,\varphi \rangle$, as 
$d \to 0$, for $\varphi \in \mathcal{C}_c(S)$. 

As before, we identify the circle with the interval $[0,1)$, that is, 
$S=[0, 1)\times [0,\iy)$. 
For $\varphi \in \mathcal{C}_c(S)$, 
\begin{align}
|\langle \mu, \varphi \rangle-\langle \mu^P,\varphi \rangle|  =&
\left|\int_0^{T}\int_{\TT}\varphi(\theta,t)\dx \mu-\sum_{k=1}^{n}\int_{T_{k-1}}^{T_k}\varphi(\xi_k(t),t)\dx m(t)\right| \nonumber \\ \nonumber
 \leq& 
\left|\int_0^{T}\int_{\TT}\varphi(\theta,t)\dx \mu-\ellcap(P)\sum_{k=1}^n\varphi(\theta_k,T_k)
\right|\\ 
+&\left|\ellcap(P)\sum_{k=1}^n\varphi(\theta_k,T_k)-\sum_{k=1}^n
\int_{T_{k-1}}^{T_k}\varphi(\xi_k(t),t)\dx m(t)\right|.
\label{measureaction}
\end{align}
The second term on the right hand side can be bounded by 
\[\sum_{k=1}^n\int_{T_{k-1}}^{T_k}
|\varphi(\theta_k,T_k)-\varphi(\xi_k(t),t)|\dx m(t).\]
 Note that
$T_k-T_{k-1}=\ellcap(P)$ and by \eqref{capvsdiam},
\[\sup_{T_{k-1} \le t <
  T_k}|\xi_k(t)-e^{2\pi i \theta_k}|\leq C \textrm{diam}(P).\]
Hence
\[
\max_{1\leq k \leq n}\, \sup_{T_{k-1}\leq t\leq T_k}|
\xi_k(t)-e^{2\pi i \theta_k}| \rightarrow 0
\]
almost surely as $d \to 0$. 
Since $\varphi$ is compactly supported, and hence uniformly continuous on $S$,
%(recall that we assumed that $\|\lambda_k \| \le M< \iy$)   
we have
\[
\max_{1\leq k \leq n}\, \sup_{T_{k-1}\leq t\leq T_k}
|\varphi(\theta_k,T_k)-\varphi(\xi_k(t),t)|<\epsilon,
\]
for $d$ sufficiently small. It follows that, almost surely,
\[\sum_{k=1}^n\int_{T_{k-1}}^{T_k}
|\varphi(\xi_k(t),t)-\varphi(\theta_k,T_{k})|\dx m(t)<
\epsilon \sum_{k=1}^n\int_{T_{k-1}}^{T_k}\dx m(t) \le c\,\epsilon\]
as soon as $d$ is sufficiently small.
%\begin{align*}
%\int_{\RR}\int_{\TT}f(\theta)g(t)\delta_{\Xi_n(t)}dm(t) &
%=\sum_{k=1}^{n}\int_{T_{k}^n}^{T_{k+1}^n}g(t)f(\theta_k)dm(t)\\
%&=\sum_{k=1}^nf(\theta_k)g(T_{k}^n)\int_{T_{k-1}^n}^{T_k^n}dm(t)+\mathcal{O}\left(\frac{1}{n}\right)\\&=\frac{1}{n}\sum_{k=1}^nX_k f(\theta_k)g(T_{k}^n)+\mathcal{O}\left(\frac{1}{n}\right)
%\end{align*}

We turn to the first term on the right hand side in \eqref{measureaction}. 
We apply the 
strong law of large numbers for independent 
random variables (see for instance \cite[Corollary 4.22]{Kallenberg})
to  
$X_{k}=\varphi(\theta_k, T_k)-\int_0^1\varphi(\theta, T_k)\dx \nu$ to 
obtain 
\[
\ellcap(P)\sum_{k=1}^nX_k \rightarrow 0 
\]
almost surely. As $\ellcap(P)=T_k-T_{k-1}$, it follows that
\[
\ellcap(P)\sum_{k=1}^n\left(\int_0^1\varphi(\theta,T_k)\dx \nu\right)
\rightarrow \int_{0}^{T}\int_0^1\varphi(\theta, t)\dx \mu,\]
almost surely, by continuity of $\varphi$.

Hence the sequence of random variables
$\langle \mu^P, \varphi\rangle$ converges 
almost surely to $\langle \mu, \varphi \rangle$ for each fixed $\varphi \in 
\mathcal{C}_c(S)$. Note that $\mathcal{C}_c(S)$ is separable. Thus, for  
$\varphi$ in a countable dense subset of $\mathcal{C}_c(S)$, we may take the 
intersection of the corresponding sets of full measure to obtain a set 
of full measure on which $\langle \mu^P, \varphi\rangle$ converges to 
$\langle \mu, \varphi\rangle$ for every $\varphi \in \mathcal{C}_c(S)$. Hence, we have
almost sure convergence of $\mu^P$ to $\mu$ with respect to the vague 
topology. Consequently, by Lemma 
\ref{convergence_in_dist} the random
measures $\mu^P$ converge almost surely to $\mu$ 
with respect to the weak topology. 

In view of Proposition 
\ref{continuity}, the corresponding conformal 
mappings converge uniformly on compact sets, and
the proof is complete.
\end{proof}

%We now turn to the case of more general building blocks.
% In fact, by examining the proof, we find that the conclusion holds for more general driving functions, 
% and hence, for more general particles. 
% \begin{corollary}\label{shape_cor}
% Let $\{\lambda_j\}_j$ and $\{\theta_j\}$ be as in the previous theorem, set 
% $T_j^n=(\sum_{j=1}^n\lambda_j)/n$ and $T=\EE[\lambda]$.

% Consider a sequence of driving functions for the Loewner equation of the form
% \[\xi_n(t)=\sum_{j=1}^n\xi_j(t),\quad \xi_j:[T_{j-1}^n,T_j^n)\rightarrow \TT,\]
% where each $\xi_j$ satisfies $\sup_{T_{j-1}^n \leq t <T_j^n}|\xi_j(t)-e^{i\theta_j}|\leq c\cdot 
% \sqrt{\lambda_j/n}$.
% Then, as $n \rightarrow \infty$, the law of the mappings $\Phi_n$, obtained by solving the Loewner equation with 
% $\xi_n$ as driving function at time $t=T_n^n$, converges weakly to that of $f_T$, the solution to the Loewner-Kufarev equation driven by $\{\nu_t\}_t=\{\nu\}_t$.
% \end{corollary}
\begin{remark}
The setup in the theorem can easily be adapted to allow for random particle 
sizes tending to zero in probability. For example, we could take 
$\ellcap(P_k^n)=\lambda_k/n$ for bounded i.i.d.\ random variables $\lambda_k$
and obtain almost sure convergence of the corresponding conformal 
mappings. The proof is essentially the same, except that we apply an
ergodic theorem
\cite[Theorem 1]{MB} instead of the law of large numbers. We can 
also relax the condition on the sequence $\lambda_k$ to square-integrability, 
and then obtain convergence in law of the conformal mappings, by adapting the
proof in \cite{MB} appropriately.
\end{remark}
\begin{remark}
Instead of choosing $\theta_j$ as i.i.d.\ random variables, one 
could also take $\{\theta_j\}_j$ to be a Markov chain satisfying 
some natural conditions. % This is done in \cite{JS}, where 
% points on the circle are chosen according 
% to harmonic measure with pole not neccessarily at $\iy$, and then 
% added at Poisson times to produce the angles $\theta_j$. 
By examining our proof, and applying a stronger version of
\cite[Theorem 1]{MB}, we obtain a result similar to Theorem 2, 
with the limiting $\nu$ uniform on $\TT$. This 
is a consequence of the fact that the invariant
measure on $\TT$ under rotation is Lebesgue measure. 
\end{remark}

%\begin{remark}
%Our assumptions on the sequence of particles are rather strong, but 
%they allow us to obtain almost sure converge of the conformal mappings.
%We believe it should be possible to relax the conditions imposed 
%on the $P_j^n$s and obtain a similar theorem, with
%almost sure convergence replaced by convergence in distribution. For 
%instance, one could consider particles whose associated capacity increments form a triangular array 
%$\{\lambda_j^n\}$ with $\max_{1\leq k\leq n}\lambda_k^n$ tending to zero 
%in probability.
%\end{remark}

\section{The evolution of harmonic measure on the cluster boundary}
In this section we establish a scaling limit for the evolution of harmonic measure on the cluster boundary. We show that it can be approximated by the solution to a deterministic ordinary differential equation related to the Loewner equation and we also characterise the stochastic fluctuations around the deterministic limit flow.

For notational simplicity, we assume that the diameters $\{d_j\}$ of the particles are constant and equal to some $d>0$ which tends to zero to obtain limit results. All the proofs can be directly adapted for $\{d_j\}$ with laws $\sigma$ with finite third moment $\sigma_3 \rightarrow 0$. We also assume that $\nu$ has density $h_{\nu}$ on $\RR$, periodic with period 1, which is twice differentiable. This restriction is purely for technical reasons and, through smoothing, any non-atomic Borel measure can be sufficiently well approximated by a measure with a twice differentiable density.

Recall the construction of the map $\gamma_P$ and the flow $(\Gamma_I: I \subseteq [0,\infty))$ from Subsection \ref{BrownianWeb}. Define the function $\beta_{\nu}$ and the constant $\rho(P)$ by
\begin{align*}
\beta_{\nu}(x) &= \int_0^1 \tilde{\gamma}_{P}(x-z) h_\nu(z) \dx z,\\
1 &= \rho(P)\int_0^1  \tilde{\gamma}_{P}(z)^2 \dx z, 
\end{align*}
where $\tilde{\gamma}_{P}^{\pm}(x) = \gamma_{P}^\pm(x)-x$. It is shown in \cite{NT} that $\rho(P) \asymp d^{-3}$.

Suppose that the Poisson process $\{T_i\}$, used in the construction of $\Gamma_I$, has rate $\ellcap(P)^{-1}$ and let $X \in D^\circ$ be a lifting of $\Gamma$ onto the real line. Then for fixed $e=(s,x) \in [0,\infty) \times \RR$,  $X_t^{e,\pm} = X_{(s,t]}^{\pm}(x)$ satisfies the integral equation
\begin{align*}
X_t^{e,\pm} &=  x + \int_{(s,t] \times [0,1)} \tilde{\gamma}_{P}^{\pm}(X_r^{e,\pm}-z)\mu(\dx r,\dx z)  \\
&= x + M_{ts}^{\pm} + \ellcap(P)^{-1} \int_{(s,t]} \beta_{\nu}(X_r^{e,\pm}){\rm d}r, \quad t \geq s
\end{align*}
where $\mu$ is a Poisson random measure (see \cite{Sato}) on $[0, \infty) \times [0,1)$, equipped with the Borel $\sigma$-algebra, with intensity $\ellcap(P)^{-1}h_\nu(z)\dx z \dx r$, and where $M_{ts}^{\pm}$ is a martingale satisfying
$$
M_{ts}^{\pm} = \int_{(s,t] \times [0,1)} \tilde{\gamma}_{P}^{\pm}(X_r^{e,\pm}-z)(\mu(\dx r,\dx z)-\ellcap(P)^{-1} h_\nu(z)\dx z \dx r). 
$$
In what follows, we suppress the superscripts $e,\pm$.

%% EDIT JS
%It is shown in \cite{NT} that there exists some universal constant $0<C_1 < \in%fty$  such that $\ellcap(P) \leq C_1 d^2$ where $\ellcap(P)=\log(\lcap(P))$ 
%is the logarithmic capacity. 
%Furthermore, if $P$ contains the set of points $\{z \in \RR: 1 < z < 1 + a \}$ f%or some constant $a>0$, then $\ellcap(P) \geq C_2 a^2$ for some universal constant $0<C_2 < \infty$. 
Recall (see Section \ref{Loewner}) that there are natural sequences of particles $P$ for which $\ellcap(P) \asymp d^2$. We assume that this holds in what follows. It is also shown in \cite{NT} that there exists some universal constant $0<C_3 < \infty$ such that
$$
\left | \int_0^1 \tilde{\gamma}_{P}(z) \dx z \right | \leq C_3 d^2,
$$
and so, by restricting to a subsequence if necessary, we assume that 
$$
\ellcap(P)^{-1} \int_0^1 \tilde{\gamma}_{P}(z) \dx z \rightarrow c_0 
$$
for some $c_0 \in \mathbb{R}$. Note that for symmetric particles, $\int_0^1 \tilde{\gamma}_{P}(z) \dx z = 0$ in which case $c_0=0$.

\begin{Pro}
As $d \rightarrow 0$, $\left | \ellcap(P)^{-1} \beta_{\nu}(x) - b(x) \right | \rightarrow 0$, uniformly in $x$, where
\begin{equation*}
b(x) =  c_0 h_\nu(x) + \frac{1}{2 \pi} \int_0^1 \cot(\pi z)(h_\nu(x-z) - h_\nu(x)) \dx z.
\end{equation*}
Furthermore, if $P$ is chosen so that
$$
d^{-1/2} \left | \ellcap(P)^{-1} \int_0^1 \tilde{\gamma}_{P}(z) \dx z - c_0 \right | \rightarrow 0,
$$
as $d \rightarrow 0$, then $d^{-1/2}\left | \ellcap(P)^{-1} \beta_{\nu}(x) - b(x) \right | \rightarrow 0$, uniformly in $x$, as $d \rightarrow 0$.
\end{Pro}
\begin{proof}
It is shown in Section 3.5 of Lawler \cite{Lawler} that there exists some universal constant $c < \infty$ such that if
$cd \leq z \leq 1 - cd$, then
\[
\left | \tilde{\gamma}_P(z) - \frac{\ellcap(P)}{2 \pi} \cot (\pi z) \right | \leq \frac{c d \ellcap(P)}{2 \pi \sin^2(\pi z)}.
\]
From this is can be deduced that there exists some $c'>0$, such that $\| \tilde{\gamma}_P\|_\infty < c'd$. 

Let us write $C(P)=\ellcap(P)$.
Now, 
$$
\beta_\nu(x) = h_\nu(x) \int_0^1 \tilde{\gamma}_P(z) \dx z + \int_0^1 \tilde{\gamma}_P(z) \left ( h_\nu(x-z) - h_\nu(x) \right ) \dx z,
$$
and so, if $d$ is sufficiently small that $cd<1$ then 
\begin{align*}
 &\left | \frac{\beta_{\nu}(x)}{C(P)} - b(x) \right |  \\
       \leq &\left | C(P)^{-1} \int_0^1 \tilde{\gamma}_P(z) \dx z - c_0 \right ||h_\nu(x)|\\
        &+ C(P)^{-1}\int_{- cd}^{cd} |\tilde{\gamma}_P(z)| |h_\nu(x-z) - h_\nu(x)| \dx z \\
        &+ C(P)^{-1}\int_{cd}^{1-cd} \left | \tilde{\gamma}_P(z) - \frac{C(P)}{2 \pi} \cot (\pi z) \right || h_\nu(x-z) - h_\nu(x) | \dx z \\
        &+ C(P)^{-1}\int_{- cd}^{cd} \frac{C(P)}{2 \pi} |\cot (\pi z)|| h_\nu(x-z) - h_\nu(x) | \dx z \\
        \leq &\|h_\nu\|_\infty \left | C(P)^{-1} \int_0^1 \tilde{\gamma}_P(z) \dx z - c_0 \right | 
        + \|  h_\nu' \|_\infty c'd C(P)^{-1}2(cd)^2 \\
        &+ \frac{cd \|h'_\nu\|_\infty \log \sin (\pi cd)}{\pi^2} 
        + \frac{cd \|h_\nu'\|_\infty}{\pi} \sup_{z \in (-cd,cd)} |z \cot (\pi z)|.
\end{align*}
\end{proof}
Note that $\int_0^1 \cot(\pi z)(h_\nu(x-z) - h_\nu(x)) \dx z$ is the Hilbert transform of $h_\nu$, as defined in Section
\ref{Loewner}. In particular, this implies that $b(x)$ is constant only when $h_\nu$ is the uniform density on the circle. It is for this reason that the behaviour in the uniform case is very different to the non-uniform case.  

Define $\phi \in D^{\circ}$ to be the solution to the ordinary differential equation
\begin{equation}
\dot{\phi}_{(s,t]}(x) = b(\phi_{(s,t]}(x)) \mbox{ for } t \geq s, \quad \phi_{(s,s]}(x) = x.
\label{floweq}
\end{equation}
We shall prove that the boundary flow converges to the flow determined by \eqref{floweq}.
Note that away from the support of $h_{\nu}$, this equation coincides
with the lifted Loewner ODE
\[
\partial_t\gamma_t(x)=\frac{1}{2\pi}\int_0^1\cot(\pi (\gamma_t(x)-z))h_{\nu}(z)dz.
\]
However, on the support of $\nu$, where we have to interpret the 
integral as a principal value, we get an additional drift term 
$c_0h_{\nu}(\gamma_t)$ in the right hand side. 
In the case of symmetric particles, the drift vanishes
everywhere, and the resulting flow is governed by the extended Loewner 
flow given by
\[\partial_t\gamma_t(x)=H[\nu](\gamma_t(x)).\]

\begin{Pro}
\label{mbounds}
 For all $T>s$, 
 $$
 \EE \left ( (\sup_{s < t < T} |M_{ts}|)^2 \right ) \leq 4 \|h_\nu \|_{\infty} \ellcap(P)^{-1}\rho(P)^{-1} (T-s).
 $$
 Hence, for all $\epsilon > 0$, 
 $$
 \PP \left ( \sup_{s < t < T} |M_{ts}| > \epsilon \right ) \rightarrow 0
 $$
 as $d \rightarrow 0$.
\end{Pro}
\begin{proof}
 Since, for any fixed $(s,x) \in [0,\infty) \times \RR$, the processes $M_{ts}$ are martingales, by Doob's $L^2$ inequality, for all $T>s$,
\begin{align*}
 \EE \left ( (\sup_{s < t < T} |M_{ts}|)^2 \right ) &\leq 4 \EE (|M_{Ts}|^2) \\
     &= 4 \int_s^T \int_0^1 \EE (\tilde{\gamma}_{P}(X_r-z)^2) \ellcap(P)^{-1} h_\nu(z)\dx z \dx r\\
     &\leq 4 \ellcap(P)^{-1} \|h_\nu\|_\infty (T-s) \int_0^1 \tilde{\gamma}_{P}(z)^2 \dx z \\
     &= 4 \|h_\nu \|_{\infty} \ellcap(P)^{-1}\rho(P)^{-1} (T-s).
\end{align*}

The second result follows from Markov's inequality and the assymptotic behaviour of $\rho(P)$ and $\ellcap(P)$.
\end{proof}

Recall the definition of $\phi$ as the solution of \eqref{floweq}.
\begin{theorem}
As $d \rightarrow 0$,
$$
d_D(X, \phi) \rightarrow 0,
$$
in probability.
\end{theorem}
\begin{proof}
 Given $\epsilon > 0$, for fixed $e=(s,x) \in [0,\infty) \times \RR$ and $T>s$, choose $d_0>0$ sufficiently small that $\|\ellcap(P)^{-1} \beta_\nu - b \|_\infty <\epsilon e^{-\|b'\|_\infty T}/2(T-s)$ for all $d \leq d_0$, and set
$$
\Omega_{T,d} = \left\{ \sup_{s < t \leq T} |M_{ts}| \leq \epsilon e^{-\|b'\|_\infty T} / 2\right\}.
$$
Then if $d \leq d_0$, on the set $\Omega_{T,d}$,
\begin{align*}
 \sup_{s<t\leq T}|X_t &- \phi_{(s,t]}(x)| \\
        \leq& \sup_{s<t\leq T} |M_{ts}| \\
         &+ \sup_{s<t\leq T} \int_s^t |\ellcap(P)^{-1}\beta_\nu(X_r) - b(X_r)|\dx r \\
         &+ \sup_{s<t\leq T} \int_s^t |b(X_r) - b(\phi_{(s,r]}(x))| \dx r \\
\leq& \epsilon e^{-\|b'\|_\infty T} + \|  b' \|_\infty \int_s^T \sup_{s<t\leq r}|X_t - \phi_{(s,t]}(x)| \dx r.
\end{align*}
Hence, by Gr\"onwall's Lemma,
$$
\sup_{s<t\leq T}|X_t - \phi_{(s,t]}(x)| \leq \epsilon.
$$
Therefore, by Proposition \ref{mbounds},
$$
\limsup_{d \rightarrow 0} \PP (\sup_{s<t\leq T}|X_t - \phi_{(s,t]}(x)| > \epsilon) \leq \limsup_{d \rightarrow 0} \PP({\Omega_{T,d}}^c) = 0.
$$
For any countable dense set $E \subset [0,\infty) \times \RR$, $X_t \rightarrow \phi_{(s,t]}(x)$ uniformly on compacts in probability as $d \rightarrow 0$, for all $(s,x) = e \in E$. Therefore, by the proof of Proposition 10.11 in \cite{NT}, $d_D(X, \phi) \rightarrow 0$ in probability as $d \rightarrow 0$.
\end{proof}

\begin{corollary}
Let $x_1,\dots,x_n$ be a positively oriented set of points in $\RR/\ZZ$ and set $x_0=x_n$.
Set $K_t=K_{\lfloor \ellcap(P)^{-1} t\rfloor}$.
For $k=1,\dots,n$, write $\omega_t^k$ for the harmonic measure in $K_t$ of the boundary segment of all
fingers in $K_t$ attached between $x_{k-1}$ and $x_k$.
Then, in the limit $d\to0$, $(\omega_t^1,\dots,\omega_t^n)_{t\geq0}$ converges weakly in $D([0,\infty),[0,1]^n)$
to $(\phi_{(0,t]}(x_1)-\phi_{(0,t]}(x_0),\dots,\phi_{(0,t]}(x_n)-\phi_{(0,t]}(x_{n-1}))_{t\geq0}$.
\end{corollary}

A geometric consequence of this result is that the number of infinite fingers of the cluster converges to the number of stable equilibria of the ordinary differential equation $\dot{x}_t = b(x_t)$, and the positions at which these fingers are rooted to the unit disk converge to the unstable equilibria of the ODE.

\subsection{Fluctuations}
In this section, suppose that $P$ is chosen so that
$$
d^{-1/2} \left | \ellcap(P)^{-1} \int_0^1 \tilde{\gamma}_{P}(z) \dx z - c_0 \right | \rightarrow 0,
$$
as $d \rightarrow 0$.
For fixed $(s,x) \in [0, \infty) \times \mathbb{R}$, define
$$
Z^{P}_t = (\ellcap(P) \rho(P))^{1/2}(X_{(s,t]}(x) - \phi_{(s,t]}(x))
$$
and let $Z_t$ be the solution to the linear stochastic differential equation
$$
\dx Z_t = \sqrt{h_\nu(\phi_{(s,t]}(x))} \dx B_t + b' (\phi_{(s,t]}(x))Z_t \dx t, \quad t \geq s,
$$
starting from $Z_s=0$, where $B_t$ is a standard Brownian motion.

Note that if $x$ is off the support of $h_\nu$, then $Z_t = 0$ for all $t \geq s$. Also observe that in the case where $\nu$ is the uniform measure on the unit circle, $(\ellcap(P) \rho(P))^{1/2}(X_{(s,t]}(x)-x-c_0(t-s))_{t \geq s}$ converges to standard Brownian motion, starting from $0$ at time $s$.

\begin{lemma}
\label{tightflow}
For fixed $x$ and $s < T < \infty$ there exists some constant $C$, dependent only on $T$, $h_\nu$ and $b$ such that   
$$
\EE \left ( \sup_{s \leq t \leq T} |Z_t^{P}|^2 \right ) \leq C
$$
and, for all $s \leq t_1 < t_2 \leq T$,
$$
\EE \left ( \sup_{t_1 \leq t \leq t_2} |Z_t^{P} - Z_{t_1}^{P} |^2 \right ) \leq C (t_2-t_1).
$$ 
Therefore the family of processes $(Z^{P}_t)_{t \geq s}$ is tight with respect to parameter $d$. 
\end{lemma}
\begin{proof}
 Since $\ellcap(P) \rho(P) \asymp d^{-1}$ there exists $C'>0$ such that 
$$
(\ellcap(P) \rho(P))^{\frac{1}{2}}\left | \ellcap(P)^{-1} \beta_{\nu}(x) - b(x) \right | < C'.
$$
Then, 
\begin{align*}
 \EE \Big ( \sup_{s \leq t \leq T} &|Z_t^{P}|^2 \Big ) \\
  \leq& 3\ellcap(P) \rho(P)\EE(\sup_{s \leq t \leq T}|M_{ts}|^2) \\
  & + 3 \ellcap(P) \rho(P) \int_s^T \EE(|\ellcap(P)^{-1} \beta_{\nu}(X_r) - b(X_r)|^2) \dx r \\
  & + 3 \ellcap(P) \rho(P) \int_s^T \EE(\sup_{s \leq t \leq r}|b(X_t) - b(\phi_{(s,t]}(x))|^2) \dx r \\
  \leq& (12 \|h_\nu\|_\infty + 3C')(T-s) + 3 \|b'\|_\infty \int_s^T \EE \left ( \sup_{s \leq t \leq r} |Z_t^{P}|^2 \right ) \dx r.
\end{align*}
The result follows by Gr\"onwall's Lemma. Similarly
\begin{align*}
 \EE \Big ( \sup_{t_1 \leq t \leq t_2} &|Z_t^{P} - Z_{t_1}^{P}|^2 \Big ) \\
  \leq&  (12 \|h_\nu\|_\infty + 3C')(t_2-t_1) \\
  & + 3 \|b'\|_\infty \int_{t_1}^{t_2} \EE (|Z_{t_1}^{P}|^2) \dx r \\
  & + 3 \|b'\|_\infty \int_{t_1}^{t_2} \EE \left ( \sup_{t_1 \leq t \leq r} |Z_t^{P} - Z_{t_1}^{P}|^2  \right ) \dx r \\
  \leq& (12 \|h_\nu\|_\infty+3C' + 3\|b'\|_\infty  \EE (|Z_{t_1}^{P}|^2)) (t_2-t_1) \\
  &  + 3 \|b'\|_\infty \int_{t_1}^{t_2} \EE \left ( \sup_{t_1 \leq t \leq r} |Z_t^{P} - Z_{t_1}^{P} |^2 \right ) \dx r.
\end{align*}
Again, the result follows by Gr\"onwall's Lemma.
\end{proof}

\begin{theorem}
\label{flucts}
As $d \rightarrow 0$, the processes $Z^{P}_t \to Z_t$ in distribution.
\end{theorem}
\begin{proof}
 Since the processes $(Z^{P}_t)_{t \geq s}$ are tight, and since both $Z^{P}_t$ and $Z_t$ have independent increments (see, for example, \cite{Kallenberg}), it is sufficient to show that, for fixed $t \geq s$, $Z^{P}_t \to Z_t$ in distribution.

 For simplicity, let $s=0$, and $x_t = \phi_{(0,t]}(x)$.

 Define $\psi_t$ to be the solution to the linear ordinary differential equation
$$
\dot{\psi_t} = -b'(x_t) \psi_t, \quad \psi_0 = 1.
$$
By It\^o's formula,
$$
\psi_t Z_t = \int_0^t \psi_s \sqrt{h_\nu(x_s)} \dx B_s \sim N \left(0, \int_0^t \psi_s^2 h_\nu(x_s) \dx s\right).
$$
Hence $Z_t$ is a Gaussian process. Similarly
$$
\psi_t Z^{P}_t = (\ellcap(P) \rho(P))^{1/2} \int_0^t \psi_s \dx M_s + \int_0^t R_s^P \dx s,
$$
where 
$$
R_t^P = (\ellcap(P) \rho(P))^{1/2}\psi_t( \ellcap(P)^{-1} \beta_\nu(X_t) - b(x_t) - b'(x_t)(X_t - x_t)).
$$
Using the bounds on $Z_t^{P}$ established above, it is straightforward to show that
$$
\int_0^t R_s^P \dx s \rightarrow 0
$$
in probability. 
Therefore it suffices to show that 
$$
(\ellcap(P) \rho(P))^{1/2} \int_0^t \psi_s \dx M_s \rightarrow N \left(0, \int_0^t \psi_s^2 h_\nu(x_s) \dx s\right)
$$
in distribution.

The characteristic function
\begin{eqnarray*}
 \chi(\eta) 
   &=& \EE \left ( \exp \left ( i \eta (\ellcap(P) \rho(P))^{1/2} \int_0^t \psi_s \dx M_s \right )\right ) \\
   &=& \EE \left ( \exp \int_0^t \zeta(\eta (\ellcap(P) \rho(P))^{1/2} \psi_s, X_s) \dx s \right ),
\end{eqnarray*} 
where
\begin{align*}
 \zeta(\theta, x)  &= \int_0^1 \left ( e^{i \theta \tilde{\gamma}_P(x - z)} - 1 - i \theta  \tilde{\gamma}_P(x-z)\right ) h_\nu(z) \ellcap(P)^{-1} \dx z \\
   &= - \frac{\theta^2}{\ellcap(P)} \int_0^1 \int_0^1 \tilde{\gamma}_P(x-z)^2(1-r)e^{i r \theta \tilde{\gamma}_P(x - z)} h_\nu(z) \dx r \dx z \\
   &= - \frac{\theta^2}{2\ellcap(P)\rho(P)} \rho(P) \int_0^1 \tilde{\gamma}_P(x-z)^2 h_\nu(z) \dx z \\
   & - \frac{\theta^2}{\ellcap(P)} \int_0^1 (1-r) \int_0^1 \tilde{\gamma}_P(x-z)^2(e^{i r \theta \tilde{\gamma}_P(x - z)}-1) h_\nu(z) \dx z \dx r.   
\end{align*}
Here we have used the fact that by It\^o's formula (see, for example, \cite{Kallenberg}), the process
$$
\left ( \exp \left ( i \theta \int_0^t \psi_s \dx M_s  - \int_0^t \zeta(\theta \psi_s, X_s) \dx s \right )\right )_{t \geq s}
$$
is a martingale.
Now
\begin{eqnarray*}
 && \left | \rho(P)\int_0^1 \tilde{\gamma}_P(x-z)^2 h_\nu(z) \dx z - h_\nu(x) \right | \\
 &\leq& \rho(P)\int_0^1 \tilde{\gamma}_P(x-z)^2|h_\nu(z) - h_\nu(x)| \dx z  \\
    &\leq& \|h_\nu'\|_\infty \rho(P) \int_0^1 \tilde{\gamma}_P(x-z)^2 |x-z| \dx z \\
    &=& \|h_\nu'\|_\infty \rho(P) \left ( \int_{-cd}^{cd} \tilde{\gamma}_P(z)^2 |z| \dx z + \int_{cd}^{1-cd}  \tilde{\gamma}_P(z)^2 z \dx z \right )\\
    &\leq& \|h_\nu'\|_\infty \rho(P) \left ( 2 (cd)^4 + \frac{9 \ellcap(P)^2}{8 \pi^4} \log \sin (\pi c d)\right ) \\
    &\rightarrow& 0,
\end{eqnarray*}
as $d \rightarrow 0$,
and
\begin{eqnarray*}
 &&\left | \int_0^1 \tilde{\gamma}_P(x-z)^2(e^{i r \theta \tilde{\gamma}_P(x - z)}-1) h_\nu(z) \dx z \right |\\
 &\leq& \|h_\nu\|_\infty |\theta| \|\tilde{\gamma}_P\|_\infty \rho(P)^{-1} \\
 &\leq& \|h_\nu\|_\infty |\theta| c'd \rho(P)^{-1}.
\end{eqnarray*}
Hence
$$
\zeta(\theta (\ellcap(P) \rho(P))^{1/2} \psi_s, X_s) \rightarrow -\theta^2 \psi_s h_\nu(x_s)
$$
in probability as $d \rightarrow 0$. Therefore
$$
 \chi(\eta) \rightarrow \exp \left ( - \eta^2 \int_0^t \psi_s^2 h_\nu(x_s) \dx s \right )
$$
and so 
$$
(\ellcap(P) \rho(P))^{1/2} \int_0^t \psi_s \dx M_s \rightarrow N \left(0, \int_0^t \psi_s^2 h_\nu(x_s) \dx s\right)
$$
in distribution, as required.
\end{proof}

\subsection{The uniform case}
In the case of non-uniform $\nu$, the behaviour of the boundary flow $(X_t)_{t \geq s}$ is dominated by non-trivial deterministic drift behaviour, and the random fluctuations only contribute as lower order perturbations. In the case when $\nu$ is the uniform measure on $[0,1)$, however, the drift vanishes and the random fluctations describe the highest order behaviour. This case is explored in detail in \cite{NT}, where it is shown that under suitable scaling, the boundary flow converges to the coalescing Brownian flow described in Section \ref{BrownianWeb}.

The key result shows that the joint distribution of flows starting from a finite collection of points in spacetime, converges to that of coalescing Brownian motions. In this subsection we give an adaptation of this proof to ${\rm HL}(0)$ clusters constructed with random diameters. We give this partly for completeness and to highlight the difference in behaviour between the uniform case and anisotropic case, but also to illustrate how all of the proofs in this section can be easily adapted to hold in the case of random diameters.

For law $\sigma$ with finite third moment $\sigma_3$, define $\rho(\sigma)$ by
$$
\rho(\sigma) \int_0^\infty \int_0^1 \tilde{\gamma}_{P(d)}(x)^2 \dx x \dx \sigma (d) = 1.
$$
Note that $\rho(\sigma)$ is well defined and $\rho(\sigma) \asymp \sigma_3^{-1}$.

Recall the construction of the flow $(\Gamma_I: I \subseteq [0,\infty))$ from Subsection \ref{BrownianWeb} (with the drift compensated for), but constructed from particles with random diameters with law $\sigma$, and with rate $\rho(P)$ replaced by $\rho(\sigma)$. Let $X \in D^\circ$ be a lifting of $\Gamma$ onto the real line. Then for fixed $e=(s,x) \in [0,\infty) \times \RR$,  $X_t = X_{(s,t]}(x)$ satisfies the integral equation
$$
X_t =  x + \int_{(s,t] \times (0, \infty) \times [0,1)} \tilde{\gamma}_{P(d)}(X_{r-}-z)\mu(\dx r, \dx d, \dx z), \quad t \geq s,
$$
where $\mu$ is a Poisson random measure of intensity $\rho(\sigma)h_\nu(z)\dx z \dx \sigma(d) \dx r$.

Write $\mu^{\sigma}_e$ for the distribution of $(X_t)_{t \geq s}$ on the Skorokhod space $D_e=D_x([s,\infty),\RR)$ of cadlag paths starting from $x$ at time $s$.
Write $\mu_e$ for the distribution on $D_e$ of a standard Brownian motion
starting from $e$.

By a straightforward adaptation of Theorem \ref{flucts}, $\mu^{\sigma}_e\to\mu_e$ weakly on $D_e$ as $\sigma_3 \rightarrow 0$.

Recall the definitions of $E$, $D_E$, $\mu_E$ from Subsection \ref{BrownianWeb}.

\begin{Pro}
We have $\mu^\sigma_E\to\mu_E$ weakly on $D_E$ as $\sigma_3 \rightarrow 0$.
\end{Pro}
\begin{proof}
The families of marginal laws $(\mu^\sigma_{e_k})$ can be shown to be tight, with respect to parameter $\sigma$ with $\sigma_3<\infty$,  by a similar argument to that in Lemma \ref{tightflow}.
Hence the family of laws $(\mu^\sigma_E)$ is also tight.
Let $\mu$ be any weak
limit law for the limit $\sigma_3 \rightarrow 0$ and let $(Z_t^{e_k})_{t \geq s_k}$, $k \in \NN$, be a sequence of limit processes. For $j,k$
distinct, the process
$$
X^{e_j}_tX^{e_k}_t-\int_{s_j\vee s_k}^tb(X^{e_j}_s,X^{e_k}_s)ds, \quad t\geq s_j\vee s_k,
$$
is a martingale, where
$$
b(x,x')=\rho(\sigma)\int_0^\infty\int_0^1\tilde {\gamma}_{P(d)}(x-z)\tilde {\gamma}_{P(d)}(x'-z)\dx z \dx \sigma(d).
$$
Let $\lambda$ be the smallest constant
$\lambda=\lambda(\sigma)\in(0,1]$ such that
$$
\rho(\sigma)\int_0^\infty\int_0^1|\tilde {\gamma}_{P(d)}(x+a)\tilde {\gamma}_{P(d)}(x)| \dx x \dx \sigma(d)\leq\lambda,\quad a\in[\lambda,1-\lambda].
$$
It is shown in \cite{NT} that, for $d$ sufficiently small, if $a \in [d^{1/4}, 1-d^{1/4}]$ then
$$
\rho(\sigma)\int_0^1|\tilde{\gamma}_{P(d)}(x+a)\tilde{\gamma}_{P(d)}(x)|\dx x
\leq d^{1/4}.
$$
For all other values of $a$, by Cauchy-Schwarz, 
$$
\rho(\sigma)\int_0^1|\tilde{\gamma}_{P(d)}(x+a)\tilde{\gamma}_{P(d)}(x)|\dx x
\leq 1.
$$
Hence if $a \in [\sigma_3^{1/24}, 1-\sigma_3^{1/24}]$, then for $\sigma_3$ sufficiently small,
\begin{align*}
\rho(\sigma)\int_0^\infty\int_0^1|\tilde {\gamma}_{P(d)}(x+a)\tilde {\gamma}_{P(d)}(x)|\dx x \dx \sigma(d)\leq& \int_0^{\sigma_3^{1/6}} d^{1/4}\dx \sigma(d)\\
&+ \PP(d > \sigma_3^{1/6}) \\
\leq& \EE(d^{1/4}) + \sigma_3/\sigma_3^{1/2} \\
\leq& \sigma_3^{1/24}.
\end{align*}
Hence $\lambda \leq C\sigma_3^{1/24} \rightarrow 0$.

We have $|b(x,x')|\leq\lambda$ whenever $\lambda\leq|x-x'|\leq 1-\lambda$. Hence, by standard
arguments, under $\mu$, the process $(Z^{e_j}_tZ^{e_k}_t:s_j\vee s_k\leq t<T^{jk})$ is a local
martingale, where $T^{jk} = \inf \{ t \geq s_j\vee s_k: |Z_t^{e_j}-Z_t^{e_k}| \notin [\lambda, 1-\lambda] \} $. We know from the proof of the proposition above that, under $\mu$, the processes
$(Z^{e_j}_t:t\geq s_j)$, $((Z^{e_j}_t)^2-t:t\geq s_j)$ and $(Z^{e_k}_t:t\geq s_k)$ are continuous local martingales.
But $\mu$ inherits from the laws $\mu^\sigma_E$ the property
that, almost surely, for all $n \in \ZZ$, the process $(Z^{e_j}_t-Z^{e_k}_t+n:t\geq s_j\vee s_k)$ does not change sign.
Hence, by an optional stopping argument, $Z^{e_j}_t-Z^{e_k}_t$ is constant for $t\geq T^{jk}$.
It follows that $(Z^{e_j}_tZ^{e_k}_t-(t-T^{jk})^+)_{t\geq s_j\vee s_k}$ is a continuous local martingale.
Hence $\mu=\mu_E$.
\end{proof}

As observed above, $X$ is a $D^\circ$-valued random variable. Let $\mu_A^\sigma$ denote the law of $X$ on the Borel $\sigma$-algebra of $D^\circ$. The following result follows immediately from the corresponding theorem in \cite{NT}.

\begin{theorem}
We have $\mu_A^\sigma\to\mu_A$ weakly on $D^\circ$ as $\sigma_3\to0$.
\end{theorem}

\begin{corollary}
Let $x_1,\dots,x_n$ be a positively oriented set of points in $\RR/\ZZ$ and set $x_0=x_n$.
Set $K_t=K_{\lfloor\rho(\sigma)t\rfloor}$.
For $k=1,\dots,n$, write $\omega_t^k$ for the harmonic measure in $K_t$ of the boundary segment of all
fingers in $K_t$ attached between $x_{k-1}$ and $x_k$.
Let $(B_t^1,\dots,B_t^n)_{t\ge0}$ be a family
of coalescing Brownian motions in $\RR/\ZZ$ starting from $(x_1,\dots,x_n)$.
Then, in the limit $\sigma_3\to0$, $(\omega_t^1,\dots,\omega_t^n)_{t\geq0}$ converges weakly in $D([0,\infty),[0,1]^n)$
to $(B_t^1-B_t^0,\dots,B^n_t-B_t^{n-1})_{t\geq0}$.
\end{corollary}

\begin{acknowledgement}
The authors thank Michel Zinsmeister and the CODY network for organizing the workshop at Universit\'e d'Orl\'eans where this work was initiated. 
We thank Michael Bj\"orklund for interesting discussions, and for providing
us with the preprint \cite{MB}. We are grateful to Bj\"orn Winckler for 
assistance in generating pictures. Johansson and Sola thank the Department of Mathematics and Statistics at Lancaster University for its hospitality during our visit. 
\end{acknowledgement}


\begin{thebibliography}{1}  

\bibitem{A79} R.A. Arratia, Coalescing {B}rownian motions on the line, PhD thesis, University of Wisconsin
1979.

\bibitem{Bauer}R.O. Bauer, Discrete L\"owner evolution,
{\itshape Ann. Fac. Sci. Toulouse Math. (6)} {\bfseries 12} (2003), 
no.4, 433--451.

\bibitem{Bil}P. Billingsley, Convergence of probability measures, 
John Wiley \& Sons, Inc., New York 1999. 

\bibitem{MB}M. Bj\"orklund, Ergodic theorems for random clusters, 
preprint 2009.

\bibitem{CM1}L. Carleson, N. Makarov, Aggregation in the plane and Loewner's equation, {\itshape Comm. Math. Phys. (216)} (2001), 583-607.

\bibitem{CM2}L. Carleson, N. Makarov, Laplacian path models. Dedicated
  to the memory of Thomas H. Wolff, {\itshape J. Anal. Math. (87)} (2002), 103-150.

\bibitem{David}B. Davidovitch, H.G.E. Hentschel, Z. Olami, I. Procaccia, L.M. Sander, 
E.Somfai, Diffusion limited aggregation and iterated conformal maps, 
{\itshape Phys. Rev. E. (87)} (1999), 1368.

\bibitem{Dieu}J. Dieudonn\'e, Foundations of Modern Analysis, 
Third printing, Academic Press, New York \& London, 1969.

\bibitem{Duren}P.L. Duren, Univalent functions, Grundlehren der  
mathematischen Wissenschaften {\bfseries 259}, Springer Verlag, 
New York 1983. 

\bibitem{Eden}M. Eden, A two-dimensional growth process. {\itshape 1961  Proc. 4th Berkeley Sympos. Math. Statist. and Prob., Vol. IV  pp. 223--239 Univ. California Press, Berkeley, Calif. }

\bibitem{FINR}L.R.G. Fontes, M. Isopi, C.M. Newman, K. Ravishankar, The {B}rownian web: characterization and convergence, {\itshape Ann. Probab. (32)} (2004), 2857--2883.

\bibitem{Garnett}J.B. Garnett, Bounded analytic functions, Graduate texts in mathematics {\bfseries 236}, Rev. first ed., Springer Verlag, 2007. 

\bibitem{HL}M. Hastings, L. Levitov, Laplacian growth as one-dimensional 
turbulence, {\itshape Phys. D. (116)} (1998), 244-252.

\bibitem{JS}F. Johansson, A. Sola, Rescaled L\'evy-Loewner hulls and 
random growth, {\itshape Bull. Sci. Math. (133)} (2009), 238-256.

\bibitem{JuKoBo}R. Julien, M. Kolb, R. Botet, Diffusion limited
  aggregation with directed and anisotropic diffusion, {\itshape
    J. Physique (45)} (1984), 395-399.

\bibitem{Kallenbergmeas}O. Kallenberg, Random measures, Third Ed., 
Akademie-Verlag Berlin, 1983.

\bibitem{Kallenberg}O. Kallenberg, Foundations of Modern Probability, Second Ed., Springer-Verlag,
New York, 2002 .

\bibitem{Lawler}G. Lawler, Conformally invariant processes in the plane,
Mathematical Surveys and Monographs {\bfseries 114}, American Mathematical 
Society, Providence, R.I. 2005.

\bibitem{LSW}G.F. Lawler, O. Schramm, W. Werner, Conformal invariance of planar loop-erased random walks and uniform spanning trees, {\itshape Ann. Probab.  (32)} (2004),  no. 1B, 939--995.

%\bibitem{Mak}N. Makarov, Fine structure of harmonic measure, {\itshape
%St. Petersburg Math. J. (10)} (1999), 217-268.

\bibitem{MS}R. Malaquias, S. Rohde, V. Sessak, M. Zinsmeister, On Laplacian 
growth, in preparation.
%Mod\`eles de croissance al\'eatoires et
%deterministes, M.Sc. Thesis, \'Ecole Polytechnique, 2006.

\bibitem{NT}J. Norris, A. Turner, Planar aggregation and 
the coalescing Brownian flow, available at http://arxiv.org/abs/0810.0211.

\bibitem{PHF}M.N. Popescu, H.G.E. Hentschel, F. Family, Anisotropic
  diffusion-limited aggregation {\itshape Phys. Rev. E (69)} (2004), no.1.

\bibitem{personalRohde}S. Rohde, Personal communication, 2008.

\bibitem{RZ}S. Rohde, M. Zinsmeister, Some remarks on Laplacian growth,
{\itshape Topology Appl. (152)} (2005), 26-43.

\bibitem{Sato}K. Sato, L\'evy processes and infinitely divisible distributions,
Cambridge Studies in Advanced Mathematics (68), Cambridge University Press, (1999).


\bibitem{TW}B. T{\'o}th, W. Werner, The true self-repelling motion,
{\itshape Probab. Theory Related Fields (111)} (1998), 375--452.

\bibitem{WS}T.A. Witten, Jr., L.M. Sander, Diffusion-limited
  aggregation, a kinetic critical phenomenon, {\itshape
    Phys. Rev. Lett. 47} (1981), 1400 - 1403.

\end{thebibliography}
\end{document}